\documentclass[a4paper,12pt,oneside]{article}
\usepackage[english]{babel}
\usepackage[T1]{fontenc} 
\usepackage[utf8]{inputenc}
\usepackage{amsthm}
\usepackage{bbm}
\usepackage{amsmath}
\usepackage{amssymb}  
\usepackage{indentfirst}
\usepackage{fancyhdr}
\usepackage{amsthm}
\usepackage{graphicx}
\usepackage{pdfpages}
\usepackage{esint}
\usepackage{url}

\numberwithin{equation}{section}

\theoremstyle{plain} 
\newtheorem{thm}{Theorem}[section] 
 
\newtheorem{lem}[thm]{Lemma} 
\newtheorem{prop}[thm]{Proposition} 
 
\newtheorem{dfn}[thm]{Definition}

\usepackage{xcolor}                    
\definecolor{custom-blue}{RGB}{0,99,166} 
\usepackage{hyperref}
\hypersetup{colorlinks=true, allcolors=custom-blue}

\usepackage{geometry}
\allowdisplaybreaks[4]

\begin{document}

\author{$\text{\sc{Antonio Giuseppe Grimaldi}}^\clubsuit$ \sc{and} $\text{\sc{Stefania Russo}}^\spadesuit$}


\title{{Regularity results for minimizers of non-autonomous integral functionals}}

\maketitle
\maketitle

\begin{abstract}
We establish the higher fractional differentiability for the minimizers of non-autonomous integral functionals of the form
\begin{equation}
 \mathcal{F}(u,\Omega):=\int_\Omega \left[ f(x,Du)- g \cdot u \right]  dx , \notag
\end{equation}
under $(p,q)$-growth conditions. Besides a suitable differentiability assumption on the partial map $x \mapsto D_\xi f(x,\xi)$, we do not need to assume any differentiability assumption on the function $g$. Moreover, we show that the higher differentiability result holds true also
assuming strict convexity and growth conditions on $f$ only at infinity.

\end{abstract}


\let\thefootnote\relax\footnotetext{
			\small $^{\clubsuit}$Dipartimento di Ingegneria, Università degli Studi di Napoli ``Parthenope'',
Centro Direzionale Isola C4, 80143 Napoli, Italy. E-mail: \textit{antoniogiuseppe.grimaldi@collaboratore.uniparthenope.it}}

\let\thefootnote\relax\footnotetext{
			\small $^{\spadesuit}$Dipartimento di Matematica e Applicazioni ``R. Caccioppoli'', Università degli Studi di Napoli ``Federico II'', Via Cintia, 80126 Napoli,
 Italy. E-mail: \textit{stefania.russo3@unina.it}}

\section{Introduction}
The aim of this paper is to present some regularity results of vectorial minimizers of variational integrals of the form
\begin{equation}\label{functional}
 \mathcal{F}(u,\Omega):=\int_\Omega \left[ f(x,Du)- g \cdot u \right]  dx .
\end{equation}
Here, $\Omega \subset \mathbb{R}^n$ is a bounded open subset, $n\geq 2$, $u : \Omega \to \mathbb{R}^N$, $N \ge 1$,
$f:\Omega\times\mathbb{R}^{N\times n}\to [0,+\infty)$, $\xi \mapsto f(x,\xi) \in \mathcal{C}^1(\mathbb{R}^{N\times n})$  and satisfies the so-called \textit{Uhlenbeck structure}, i.e.\ $f$ is represented in
the form $f (x, \xi) = \tilde{f}(x, |\xi|)$ for a given function $\tilde{f} :\Omega\times[0,+\infty)\to [0,+\infty)$. 
Besides, we assume that there exist positive constants $\nu, L, M, l$ and exponents $2 \leq p < q < \infty$ such that

\begin{equation}\label{F1}
  \nu|\xi|^p \leq f(x,\xi) \leq L(1+|\xi|^q)  \tag{F1}
\end{equation}
\begin{equation}\label{F2}
    \langle D_\xi f(x, \xi)-D_\xi f(x, \eta),\xi-\eta\rangle  \ge M (1 +|\xi|^2+|\eta|^2)^\frac{p-2}{2}|\xi - \eta|^2 \tag{F2}
\end{equation}
\begin{equation}\label{F3}
     |D_\xi f(x, \xi)-D_\xi f(x, \eta)| \le l (1 +|\xi|^2+|\eta|^2)^\frac{q-2}{2}|\xi - \eta| \tag{F3}
\end{equation}

\noindent for a.e.\ $x \in \Omega$ and all $\xi,\eta \in \mathbb{R}^{N \times n}$.
Concerning the dependence on the $x$-variable, we assume that
there exists a non-negative function 
$k: \Omega \to [0, +\infty)$ such that
\begin{equation}\label{F4}
    |D_\xi f(x,\xi)-D_\xi f(y, \xi)| \leq |x-y|^{\gamma} (k(x)+k(y)) (1 +|\xi|^2)^\frac{q-1}{2} \tag{F4}
\end{equation}
\noindent for a.e.\ $x,y \in \Omega$ and every $\xi \in \mathbb{R}^{N \times n}$, where $\gamma \in (0,1)$.

It is worth noting that assumption \eqref{F4} with $k \in L^r_{loc}(\Omega)$ states that the partial map $x \mapsto D_\xi f(x,\xi)$ belongs to the Besov space $B_{r,\infty}^{\gamma}(\Omega, \mathbb{R}^{N \times n})$ (for precise definition and properties of Besov spaces see Section \ref{secbesov}).
\\On the other hand, we say that assumption \eqref{eqf5} is satisfied if there exists a sequence of measurable non-negative functions $g_k \in L^{r}_{loc}(\Omega)$ such that
$$\displaystyle\sum_{k=1}^{\infty} \Vert g_k \Vert^{s}_{L^{r}(\Omega)} < \infty,$$
for some exponent $s \ge 1$, and at the same time
\begin{equation}\tag{F5}
    |D_{\xi}f(x,\xi)-D_{\xi}f(y, \xi)| \leq |x-y|^{\gamma} (g_k(x)+g_k(y)) (1 +|\xi|^2)^{\frac{q-1}{2}} \label{eqf5}
\end{equation}

\noindent for a.e.\ $x,y \in \Omega$ such that $2^{-k} \text{diam}(\Omega) \leq |x-y| < 2^{-k+1}\text{diam}(\Omega)$ and for every $\xi \in \mathbb{R}^{N \times n}$. As before, we remark that assumption \eqref{eqf5} implies that the partial map $x \mapsto D_\xi f(x,\xi)$ belongs to the Besov space $B^\gamma_{r,s}(\Omega, \mathbb{R}^{N \times n})$.

Let us recall the definition of local minimizer.
\begin{dfn}
{\em A function $u\in W_{\rm loc}^{1,1}(\Omega, \mathbb{R}^N)$ is a  local minimizer of
\eqref{functional} if, for every open subset $\tilde{\Omega} \Subset \Omega$, we have $\mathcal{F}(u, \tilde{\Omega}) <  \infty$ and  
$\mathcal{F}(u;\tilde{\Omega})\le  \mathcal{F}(\varphi;\tilde{\Omega})$ holds
for all $\varphi\in u+W_0^{1,1}(\tilde{\Omega},\mathbb{R}^N)$.}
\end{dfn}
We will say that a function $F$ has \textit{$(p, q)$-growth conditions} if assumption \eqref{F1} is in force. 
The study of regularity properties of local minima to functionals with general growth started with the pioneering papers by Marcellini \cite{mar}, see also \cite{mar91, mar93}. 

When referring to $(p, q)$-growth conditions \eqref{F1},
we call the quantity $q/p > 1$ \textit{the gap ratio of the integrand} $F$, or simply, \textit{the gap}.
A main point is that in order to get regularity of minimizers to non-autonomous  functionals with non-standard growth conditions, even the local boundedness, a restriction between $p$ and $q$ need to be imposed, usually expressed in the form 
\begin{equation*}
    q \le c(n,r)p, \qquad \text{with } c(n,r ) \to 1 \ \text{as} \ r \to n,
\end{equation*}
where here $r$ is the exponent appearing in \eqref{F4}-\eqref{eqf5} and usually denotes the degree of regularity of the map $x \mapsto D_\xi f(x, \xi)$.
We refer to \cite{giaquinta,mar91} for counterexamples.

In order to explain the difficulties of our problem, we now recall some class of functionals included in our setting and some known regularity properties of the related minimizers.


It is well known that no extra differentiability properties for minimizers to 
\begin{equation}
 \mathcal{G}(u,\Omega):=\int_\Omega  f(x,Du) \ dx \notag
\end{equation}
can be expected, unless some assumption is given on the coefficients of the operator $D_\xi f(x,\xi)$. A $W^{1,r}$ Sobolev regularity, with $r \ge n$, on the partial map $x \mapsto D_\xi f(x,\xi)$ is a sufficient condition for the higher differentiability of minima (see \cite{elpdn,elpdn2,gavioli1,gavioli2,gentile,gi,gpdn,grimaldi,pdnacv,pdnpa}). In the case of standard $p$-growth, the higher fractional differentiability of minimizers has been proved  assuming that the coefficients of $ D_\xi f(x,\xi)$ belong to the Besov space $B^\alpha_{\frac{n}{\alpha},s}$ (see \cite{baison.clop2017,baison,clop,cruz,kristensen.mingione}). Regarding the non-standard growth, in \cite{grimaldi.ipocoana,grimaldi.ipocoana1} this analysis has been carried out in the setting of obstacle problems, assuming that the map $x \mapsto D_\xi f(x,\xi)$ satisfies a $B^\alpha_{r,s}$-regularity for $r > \frac{n}{\alpha}$ (see \cite{gripo} for the unconstrained case). It is worth noticing that solutions $u$ to obstacle problems solve elliptic equations of the form
\begin{equation}
    \text{div} D_\xi f(x,Du) = \text{div} D_\xi f(x,D \psi)=: \bar{g},
\end{equation}
where $\psi$ is the obstacle function and $u \ge \psi$. In \cite{grimaldi.ipocoana,grimaldi.ipocoana1}, the higher differentiability has been obtained assuming that $D \psi \in W^{1,2q-p}(\Omega)$, that implies $\bar{g} \in L^{2q-p}(\Omega)$. Here, in Theorem \ref{mainthm} below, we will prove some higher differentiability properties of minimizers under weaker assumptions on the right-hand side. 

Actually, the research in this field is so intense that it is almost impossible to give an exhaustive and comprehensive list of references; in addition to those mentioned in this paper and the references therein, we confine ourselves to refer the interested reader to the survey  \cite{MingSurvey}. \\
In particular in \cite{cgpdn}, the authors proved the extra fractional differentiability of weak solutions of the following nonlinear elliptic equations in divergence form
$$\text{div} \mathcal{A}(x,Du)= \text{div } G \quad \text{in } \Omega,$$
where the operator $\mathcal{A}$ satisfies \eqref{F2}--\eqref{F4} with $p=q$. It has been proved that a higher differentiability on the partial map $x \mapsto \mathcal{A}(x, \cdot) $ and on the right-hand side $G$ transfer to the function $V_p(Du)$ (see Section \ref{Preliminary} for the definition).

{The main novelty of our results is that we consider $(p,q)$-growth, which means we are able to extend the results of \cite{cgpdn} to the context of non-standard growth. Moreover, we show that the gradient of minimizers to \eqref{functional} posseses some fractional differentiability properties, without assuming any differentiability properties on the function $g$, but only  a suitable order of integrability in the setting of Lebesgue spaces.

Our first result is the following
\begin{thm}\label{mainthm}
Let $f$ satisfy \eqref{F1}--\eqref{F4} and let  $g \in L^{m}_{loc}(\Omega, \mathbb{R}^N)$, with $p' \leq m \leq 2$, for exponents $2\leq p  <\frac{n}{\lambda}< r$, $p<q$ such that
\begin{equation}\label{gap.}
\dfrac{q}{p} <  1+ \dfrac{ \lambda}{n}- \frac{1}{r} ,
\end{equation}
where $\lambda:= \min \{\gamma, \frac{m}{2} \}$.
Let $u\in W^{1,p}_{loc}(\Omega, \mathbb{R}^N)$ be a local minimizer of \eqref{functional}. Then, $V_p(Du)  \in B^\lambda_{2,\infty,\textrm{\text{loc}}}(\Omega, \mathbb{R}^{N \times n})$
and the following estimate holds
\begin{align}
   \int_{B_{R/4}} |\tau_h V_p (Du)|^2 dx \leq  c |h|^{2 \lambda} \left( \int_{B_{R}} (1 + |Du|^p) dx + \Vert k \Vert_{L^r (B_{R})} +  \Vert g \Vert_{L^{m}(B_R)} \right)^\sigma,
\end{align}
for every concentric balls $B_{R/4} \subset B_{R} \Subset \Omega$, where $c = c(n, p,q,m, L,l,  \nu,r,  R)$ and $\sigma= \sigma(n,p,q,m, \gamma,r)$ are positive constants.
\end{thm}

One of the main motivations to the study of functionals with $(p,q)$-growth comes from the applications, for instance to the theory of elasticity for strongly anisotropic materials (see Zhikov \cite{Z1,Z2}). A very well known model is given by the so-called \textit{double phase functional} defined by
\begin{equation}
   \int_\Omega \left[ |Du|^p+a(x)|Du|^q \right] dx, \label{DP}
\end{equation}
where $a(x)$ is non-negative and H\"older continuous with exponent $\alpha$. The regularity properties of local minimizers to such functional have been widely investigated in \cite{BCM,25dfm,FMM} with the aim of identifying sufficient and necessary conditions on the relation between $p$, $q$ and $\alpha$ to establish the H\"older continuity of the gradient of the local minimizers. 

Very recently, in \cite{Inventiones,DFM} it has been proved the local gradient H\"older continuity of minimizers of non-uniformly elliptic integrals, that
are not necessarily equipped with a Euler-Lagrange equation due to the mere local H\"older continuity of coefficients, under the gap 
\begin{equation}
    \dfrac{q}{p} < 1 + \dfrac{\alpha}{n}. \label{GAP.}
\end{equation}
We observe that in the model case \eqref{DP}, the inequality \eqref{gap.}
gives back \eqref{GAP.}. Indeed, by Lemma \ref{EmB} below, if $a(x) \in B^\lambda_{r,\infty}(\Omega)$ then $a(x) \in \mathcal{C}^{0,\alpha}(\Omega)$
with exponent
\begin{equation*}
 \alpha=\lambda-\dfrac{n}{r}.
\end{equation*}
Moreover, in \cite{Inventiones} the authors proved the H\"older continuity of $Du$ assuming that $g \in L^{n/\alpha,1/2}(\Omega)$. However, in Theorem \ref{mainthm} we impose a weaker summability on the function $g$, that is $g \in L^{m}_{loc}(\Omega, \mathbb{R}^N)$ for $p' \leq m \leq 2$, and so our results are not covered by those in \cite{Inventiones}.

Now assuming \eqref{eqf5} instead of \eqref{F4} and arguing in a similary way, we obtain the following theorem.
\begin{thm}\label{thmBfinito}
Let $f$ satisfy \eqref{F1}--\eqref{F3}, \eqref{eqf5} and let  $g \in L^{m}_{loc}(\Omega, \mathbb{R}^N)$, with $p' \leq m \leq 2$, for exponents $2\leq p  <\frac{n}{\lambda}< r$, $p<q$ such that
\eqref{gap.} holds.
Let $u\in W^{1,p}_{loc}(\Omega, \mathbb{R}^N)$ be a local minimizer of \eqref{functional}. Then, $V_p(Du)  \in B^\lambda_{2,s,\textrm{\text{loc}}}(\Omega, \mathbb{R}^{N \times n})$
and the following estimate holds
\begin{align}
  \displaystyle\int_{B_{1}} \biggl( \displaystyle\int_{B_{R/4}} \dfrac{|\tau_hV_p (Du)|^{2}}{|h|^{2 \lambda }} dx \biggr)^{\frac{s}{2}}  \dfrac{dh}{|h|^{n}} \leq c\left( \int_{B_{R}} (1 + |Du|^p) dx + \sum_{k=1}^{\infty }\Vert g_k \Vert_{L^r (B_{R})} + \Vert g \Vert_{L^{m}(B_R)} \right)^\sigma, 
\end{align}
for every concentric balls $B_R \subset B_{2R} \Subset \Omega$, where $c = c(n, p,q,m, L,l,  \nu,r,  R)$, $\sigma= \sigma(n,p,q,m, \gamma,r)$ and $\lambda:=  \min \{\gamma, \frac{m}{2} \}$ are positive constants.
\end{thm}

Furthermore, we want to analyse the higher differentiability when the integrand $f$ satisfies strict convexity only at infinity, more precisely we assume that there exist positive constants $\nu, L, M, l$ and exponents $2 \leq p < q < \infty$ and $\gamma \in (0,1)$ such that

\begin{equation}\label{H1}
  \nu|\xi|^p \leq f(x,\xi) \leq L(1+|\xi|^q)  \tag{H1}
\end{equation}
\begin{equation}\label{H2}
    \langle D_\xi f(x, \xi)-D_\xi f(x, \eta),\xi-\eta\rangle  \ge M (|\xi|^2+|\eta|^2)^\frac{p-2}{2}|\xi - \eta|^2 \tag{H2}
\end{equation}
\begin{equation}\label{H3}
     |D_\xi f(x, \xi)-D_\xi f(x, \eta)| \le l (|\xi|^2+|\eta|^2)^\frac{q-2}{2}|\xi - \eta| \tag{H3}
\end{equation}
\begin{equation}\label{H4}
    |D_\xi f(x,\xi)-D_\xi f(y, \xi)| \leq |x-y|^{\gamma} (k(x)+k(y)) |\xi|^{q-1} \tag{H4}
\end{equation}
\noindent for a.e.\ $x,y \in \Omega$ and  every $\xi, \eta \in \mathbb{R}^{N \times n}$ with $\lvert \xi \rvert , \lvert \eta \rvert \geq 1$. \\
This class of functionals was first studied by Chipot and Evans in the pioneering paper \cite{CE}. A model case of a functional satisfying previous assumptions is given by
\begin{equation}
    \int_{\Omega} \Big( (|Du|-1)_+^p \ + a(x)(|Du|-1)_+^q \Big) \ dx,\notag
\end{equation}
with a non-negative and bounded function $a(x)$. Its Euler–Lagrange equation naturally arises as a model for optimal transport problems with congestion
effects. We refer to \cite{Bra1,Bra2,CJS} and reference therein for a
detailed derivation of the model.

In recent years there has been
a considerable of interest in the study of higher differentiability of solutions to widely degenerate equations that behaves as the $p$-Laplace operator at infinity, see for instance \cite{Ambrosio,AGPass,Bra1,Bra2,Clop2,Russo}. We also quote the papers \cite{Cupini1, Cupini2, EMM1} where the authors studied asymptotically convex functionals satisfying $(p,q)$-growth conditions with Sobolev coefficients. \\ 
Here, we continue this analysis in a more general setting. Indeed, the main novelty is the treatment of widely degenerate functionals that satisfy non-standard growth and allow a Besov regularity on the $x$-variable. 

When studying the regularity of the solutions to widely degenerate problems, one needs to establish \textit{good} a priori estimates and then find a family of non-degenerate approximating problems, whose solutions posses the regularity needed to establish the a priori estimates. Exploiting the higher differentiability result in Theorem \ref{mainthm}, we are able to prove the following

\begin{thm}\label{WDT}
   Let $f$ satisfy \eqref{F1*}--\eqref{F4*}, and let $g \in L^{m}_{loc}(\Omega, \mathbb{R}^N)$, with $p' \leq m \leq 2$, for exponents $2\leq p  <\frac{n}{\lambda}< r$, $p<q$ such that
\eqref{gap.} holds, where $\lambda:= \min \{\gamma, \frac{m}{2} \}$.
Then, {every local minimizer $u\in W^{1,p}_{loc}(\Omega, \mathbb{R}^N) $ to \eqref{functional} is such that
$V_p(Du)   \in B^\lambda_{2,\infty,loc}(\Omega, \mathbb{R}^{N \times n})$.}
\end{thm}

We point out that the lack of global uniform convexity of the integrand $f$ yields a lack of uniqueness of the local minimizers to \eqref{functional}. Therefore, the sequence of the local minimizers of the approximating functionals converges to a regular minimizer, but we could not say that all the minimizers to \eqref{functional} have the desired regularity. In order to overcome this problem, we introduce in the approximating problems a penalization term that does not effect the a priori estimate and that forces the approximating sequence to converge to an arbitrarily fixed local minimizer (see \cite{Cupini1, Cupini2}).

We briefly describe the structure of the paper. After summarizing some known results and fixing few notation, we recall an approximation lemma in Section \ref{AppSec}. Section \ref{apriorisec} is devoted to the proof of some a priori estimates for solutions to a family of approximating problems. Next, in Section \ref{mainthmsec}, we pass to the limit in the approximating problems. Finally, in Section \ref{WDSec}, we study the widely degenerate case.

\section{Preliminaries}\label{Preliminary}
In this section we introduce some notations and collect several results that we shall use to establish our main result.

We will follow  denote by $c$ or $C$ a general constant that may vary on different occasions, even within the same line of estimates. Relevant dependencies on parameters and special constants will be suitably emphasized using parentheses or subscripts.

In what follows, $B(x,r)=B_{r}(x)= \{ y \in \mathbb{R}^{n} : |y-x | < r  \}$ will denote the ball centered at $x$ of radius $r$. We shall omit the dependence on the center and on the radius when no confusion arises.

We define an auxiliary function by
\begin{center}
$V_{p}(\xi):=( 1 +|\xi|^{2})^\frac{p-2}{4} \xi $
\end{center}
for all $\xi\in \mathbb{R}^{m}$, $m \in \mathbb{N}$. 
For the function $V_{p}$, we recall the following estimate (see e.g. \cite[Lemma 8.3]{giusti}). 
\begin{lem}\label{D1}
Let $1<p<+\infty$. There exists a constant $c=c(n,p)>0$ such that
\begin{center}
$c^{-1}(1+|\xi|^{2}+|\eta|^{2})^{\frac{p-2}{2}} \leq \dfrac{|V_{p}(\xi)-V_{p}(\eta)|^{2}}{|\xi-\eta|^{2}} \leq c(1+|\xi|^{2}+|\eta|^{2})^{\frac{p-2}{2}} $
\end{center}
for any $\xi, \eta \in \mathbb{R}^{m}$, $\xi \neq \eta$.
\end{lem}
Now we state a well-known iteration lemma (see \cite[Lemma 6.1]{giusti} for the proof).
\begin{lem}\label{lm2}
Let $\Phi  :  [\frac{R}{2},R] \rightarrow \mathbb{R}$ be a bounded nonnegative function, where $R>0$. Assume that for all $\frac{R}{2} \leq r < s \leq R$ it holds
$$\Phi (r) \leq \theta \Phi(s) +A + \dfrac{B}{(s-r)^2}+ \dfrac{C}{(s-r)^{\gamma}}$$
where $\theta \in (0,1)$, $A$, $B$, $C \geq 0$ and $\gamma >0$ are constants. Then there exists a constant $c=c(\theta, \gamma)$ such that
$$\Phi \biggl(\dfrac{R}{2} \biggr) \leq c \biggl( A+ \dfrac{B}{R^2}+ \dfrac{C}{R^{\gamma}}  \biggr).$$
\end{lem}

\subsection{Difference quotient}
\label{secquo}
We recall some properties of the finite difference quotient operator that will be needed in the sequel. 

\begin{dfn}
Let $F$ be a function defined in an open set $\Omega \subset \mathbb{R}^n$ and let $h \in \mathbb{R}^n$. We call the difference quotient of $F$ with respect to $h$ the function
$$ 
\tau_{h}F(x) :=F(x+h)-F(x) .$$
\end{dfn}
The function $\tau_{h}F$ is defined in the set
$$\tau_{h}\Omega := \{  x \in \Omega : x+h \in \Omega \},$$
and hence in the set
$$\Omega_{|h|}: = \{ x \in \Omega : \mathrm{dist}(x,\partial \Omega)> |h|  \}.$$

We start with the description of some elementary properties that can be found, for example, in \cite{giusti}.
\begin{prop}\label{rapportoincrementale}
Let $F \in W^{1,p}(\Omega)$, with $p \geq1$, and let $G:\Omega \rightarrow \mathbb{R}$ be a measurable function.
Then
\\(i) $\tau_{h}F \in W^{1,p}(\Omega_{|h|})$ and 
$$D_{i}(\tau_{h}F)=\tau_{h}(D_{i}F).$$
(ii) If at least one of the functions $F$ or $G$ has support contained in $\Omega_{|h|}$, then
$$\displaystyle\int_{\Omega}F \tau_h G   dx = -\displaystyle\int_{\Omega} G \tau_{-h}F dx.$$
(iii) We have $$\tau_{h} (FG)(x)= F(x+h)\tau_{h} G(x)+G(x) \tau_{h} F(x).$$
\end{prop}
The next result about the finite difference operator is a kind of integral version of Lagrange Theorem.
\begin{lem}\label{ldiff}
If $0<\rho<R,$ $|h|<\frac{R-\rho}{2},$ $1<p<+\infty$ and $F\in W^{1,p}(B_{R})$, then
\begin{center}
$\displaystyle\int_{B_{\rho}} |\tau_{h}F(x)|^{p} dx \leq c(n,p)|h|^{p} \displaystyle\int_{B_{R}} |DF(x)|^{p} dx$.
\end{center}
Moreover,
\begin{center}
$\displaystyle\int_{B_{\rho}} |F(x+h)|^{p} d x \leq  \displaystyle\int_{B_{R}} |F(x)|^{p}d x$.
\end{center}
\end{lem}

We recall a fractional version of Sobolev embedding property.
\begin{lem}\label{EmbendMigliore}
Let $F \in L^2(B_R)$. Suppose that there exist $\rho \in (0,R)$, $0 < \alpha < 1$ and $M >0$ such that $$  \displaystyle\int_{B_{\rho}}|\tau_{h}F(x)|^2 dx \leq M^2 |h|^{2 \alpha}, $$ for every $h$ such that $|h| < \frac{R-\rho}{2}$. Then $F \in L^{\frac{2n}{n-2 \beta}}(B_{\rho})$ for every $\beta \in (0, \alpha)$ and $$\Vert  F \Vert _{L^{\frac{2n}{n-2 \beta}}(B_{\rho})} \leq c (M+ \Vert F \Vert _{L^2(B_R)}),  $$ with $c=c(n,R,\rho, \alpha, \beta)$.
\end{lem}

\subsection{Besov spaces}
\label{secbesov}
We give the definition of Besov spaces as done in \cite[Section 2.5.12]{haroske}.
\begin{dfn}
 Let $1 \le p < +\infty$ and $0<\alpha <1 $. 
Given $1 \leq s< +\infty$,  we say that a function $v : \mathbb{R}^n \rightarrow \mathbb{R}^m$, $m \in \mathbb{N}$, belongs to the Besov space $B^{\alpha}_{p,s}(\mathbb{R}^{n})$ if, and only if, $v \in L^{p}(\mathbb{R}^{n})$ and
 \begin{center}
$[v]_{B^{\alpha}_{p,s}(\mathbb{R}^{n})}: =  \biggl( \displaystyle\int_{\mathbb{R}^{n}} \biggl( \displaystyle\int_{\mathbb{R}^{n}} \dfrac{|\tau_hv(x)|^{p}}{|h|^{\alpha p}} dx \biggr)^{\frac{s}{p}}  \dfrac{dh}{|h|^{n}} \biggr)^{\frac{1}{s}} < + \infty  $.
\end{center}
Equivalently, we could simply say that $v \in L^{p}(\mathbb{R}^{n})$ and $\frac{\tau_{h}{v}}{|h|^{\alpha}} \in L^{s}\bigl( \frac{dh}{|h|^{n}}; L^{p}(\mathbb{R}^{n}) \bigr)$.
\\When $s = + \infty$, the Besov space $B^{\alpha}_{p,\infty}(\mathbb{R}^{n})$ consists of functions $v \in L^{p}(\mathbb{R}^{n})$ such that
\begin{center}
$[v]_{B^{\alpha}_{p, \infty}(\mathbb{R}^{n})} :=  \displaystyle\sup_{h \in \mathbb{R}^{n}} \biggl( \displaystyle\int_{\mathbb{R}^{n}} \dfrac{|\tau_hv(x)|^{p}}{|h|^{\alpha p}} dx \biggr)^{\frac{1}{p}} < +\infty $.
\end{center}

\noindent Accordingly, for $1 \le s \le + \infty$, the Besov space $B^{\alpha}_{p,s}(\mathbb{R}^{n})$ is normed with
\begin{center}
$\Vert v \Vert_{B^{\alpha}_{p,s}(\mathbb{R}^{n})}: = \Vert v \Vert_{L^{p}(\mathbb{R}^{n})} + [v]_{B^{\alpha}_{p,s}(\mathbb{R}^{n})} $.
\end{center}
\end{dfn}
Observe that, if $1 \le s < + \infty$, integrating for $h \in B(0, \delta)$ for a fixed $\delta >0$ then an equivalent norm is obtained, because
\begin{center}
$\biggl( \displaystyle\int_{\{|h| \geq \delta\}} \biggl( \displaystyle\int_{\mathbb{R}^{n}} \dfrac{|\tau_h v(x)|^{p}}{|h|^{\alpha p}} dx \biggr)^{\frac{s}{p}}  \dfrac{dh}{|h|^{n}} \biggr)^{\frac{1}{s}} \leq c(n, \alpha,p,s, \delta) \Vert v \Vert_{L^{p}(\mathbb{R}^{n})} $.
\end{center}
In the case $s = + \infty$, one can simply take supremum over $|h| \leq \delta$ and obtain an equivalent norm. By construction, one has $B^{\alpha}_{p, s}(\mathbb{R}^{n}) \subset L^{p}(\mathbb{R}^{n})$. One also has the following version of Sobolev embeddings (a proof can be found at \cite[Proposition 7.12]{haroske}).
\begin{lem}\label{3.1}
Suppose that $0 < \alpha <1$.
\\ (a) If $1 < p < \frac{n}{\alpha}$ and $1 \leq s \leq p^{*}_{\alpha} = \frac{np}{n- \alpha p}$, then there is a continuous embedding $B^{\alpha}_{p, s}(\mathbb{R}^{n}) \subset L^{p^{*}_{\alpha}}(\mathbb{R}^{n})$.
\\ (b) If $p = \frac{n}{\alpha}$ and $1 \leq s \leq + \infty$, then there is a continuous embedding $B^{\alpha}_{p, s}(\mathbb{R}^{n}) \subset BMO(\mathbb{R}^{n})$,
\\ where $BMO$ denotes the space of functions with bounded mean oscillations \emph{\cite[Chapter 2]{giusti}}.
\end{lem}
We recall the following inclusions between Besov spaces (\cite[Proposition 7.10 and Formula (7.35)]{haroske}).
\begin{lem}\label{3.2}
Suppose that $0 < \beta < \alpha < 1 $.
\\ (a) If $1 < p < +\infty$ and $1 \leq s \leq t \leq + \infty$, then $B^{\alpha}_{p, s}(\mathbb{R}^{n}) \subset B^{\alpha}_{p, t}(\mathbb{R}^{n})$.
\\ (b) If $1 < p < +\infty$ and $1 \leq s , t \leq + \infty$, then $B^{\alpha}_{p, s}(\mathbb{R}^{n}) \subset B^{\beta}_{p, t}(\mathbb{R}^{n})$.
\\ (c) If $1 \le s \le + \infty$, then $B^{\alpha}_{\frac{n}{\alpha}, s}(\mathbb{R}^{n}) \subset B^{\beta}_{\frac{n}{\beta}, s}(\mathbb{R}^{n})$.
\end{lem}

Combining Lemmas \ref{3.1} and \ref{3.2}, we get the following Sobolev type embedding theorem for Besov spaces $B^\alpha_{p,\infty}(\mathbb{R}^n)$.

\begin{lem}\label{besovembed}
Suppose that $0 < \alpha < 1$ and $1 < p < \frac{n}{\alpha}$. There is a continuous embedding $B^\alpha_{p,\infty}(\mathbb{R}^n) \subset L^{p^*_\beta}(\mathbb{R}^n)$, for every $0 < \beta < \alpha$. Moreover, for every function $F \in B^\alpha_{p,\infty}(\mathbb{R}^n)$ the following local estimate
\begin{equation}
\Vert F \Vert _{L^{\frac{np}{n-\beta p}}(B_\varrho)} \leq \dfrac{c}{(R-\varrho)^{\delta}} (\Vert F \Vert_{L^{p}(B_R)} + [F]_{B^{\alpha}_{p,q}(B_R)})
\end{equation}
holds for every ball $B_\varrho \subset B_R$, with $c=c(n,p,q,\alpha,\beta)$ and $\delta=\delta(n,p,q)$.
\end{lem}

Given a domain $\Omega \subset \mathbb{R}^{n}$, we say that a function $v: \mathbb{R}^n \rightarrow \mathbb{R}^m$ belongs to the local Besov space $ B^{\alpha}_{p, s,\text{loc}}$ if $\varphi  v \in B^{\alpha}_{p, s}(\mathbb{R}^{n})$ whenever $\varphi \in \mathcal{C}^{\infty}_{0}(\Omega)$. It is worth noticing that one can prove suitable version of Lemmas \ref{3.1} and \ref{3.2}, by using local Besov spaces.

We have the following lemma (we refer to \cite[Lemma 7]{baison.clop2017} for the proof).
\begin{lem}
Let $1 \le p < + \infty$, $1 \le s \le + \infty$ and $0< \alpha < 1$.
A function $v \in L^{p}_{\text{loc}}(\Omega)$ belongs to the local Besov space $B^{\alpha}_{p,s,\text{loc}}$, if, and only if,
\begin{center}
$\biggl\Vert \dfrac{\tau_{h}v}{|h|^{\alpha}} \biggr\Vert_{L^{s}\bigl(\frac{dh}{|h|^{n}};L^{p}(B)\bigr)}< + \infty,$
\end{center}
for any ball $B\subset2B\subset\Omega$ with radius $r_{B}$. Here the measure $\frac{dh}{|h|^n}$ is restricted to the ball $B(0,r_B)$ on the h-space.
\end{lem}

It is known that Besov spaces of fractional order $\alpha \in (0,1)$ can be characterized in pointwise terms. We give the following definition according to \cite{koskela}.
\begin{dfn}
 Given a measurable function $v:\mathbb{R}^{n} \rightarrow \mathbb{R}^m$, a \textit{fractional $\alpha$-Hajlasz gradient for $v$} is a sequence $\{g_{k}\}_{k}$ of measurable, non-negative functions $g_{k}:\mathbb{R}^{n} \rightarrow \mathbb{R}$, together with a null set $A\subset\mathbb{R}^{n}$, such that the inequality 
\begin{center}
$|v(x)-v(y)|\leq (g_{k}(x)+g_{k}(y))|x-y|^{\alpha}$
\end{center} 
holds whenever $k \in \mathbb{Z}$ and $x,y \in \mathbb{R}^{n}\setminus A$ are such that $2^{-k} \leq|x-y|<2^{-k+1}$. We say that $\{g_{k}\}_{k} \in l^{s}(\mathbb{Z};L^{p}(\mathbb{R}^{n}))$ if
\begin{center}
$\Vert \{g_{k}\}_{k} \Vert_{l^{s}(L^{p})}=\biggl(\displaystyle\sum_{k \in \mathbb{Z}}\Vert g_{k} \Vert^{s}_{L^{p}(\mathbb{R}^{n})} \biggr)^{\frac{1}{s}}<  +\infty.$
\end{center} 
\end{dfn}
The following result was proved in \cite{koskela}.
\begin{thm}
Let $0< \alpha <1,$ $1 \leq p < +\infty$ and $1\leq s \leq +\infty $. Let $v \in L^{p}(\mathbb{R}^{n})$. One has $v \in B^{\alpha}_{p,s}(\mathbb{R}^{n})$ if, and only if, there exists a fractional $\alpha$-Hajlasz gradient $\{g_{k}\}_{k} \in l^{s}(\mathbb{Z};L^{p}(\mathbb{R}^{n}))$ for $v$. Moreover,
\begin{center}
$\Vert v \Vert_{B^{\alpha}_{p,s}(\mathbb{R}^{n})}\simeq \inf \Vert \{g_{k}\}_{k} \Vert_{l^{s}(L^{p})},$
\end{center}
where the infimum runs over all possible fractional $\alpha$-Hajlasz gradients for $v$.
\end{thm}

We conclude this section with the following embedding result (see \cite[Section 2.7.1, Remark 2.]{Triebel}).
\begin{lem}\label{EmB}
Let $0 < \alpha < 1$ and $  1 \le s \le + \infty$. If $p > \frac{n}{\alpha}$, then we have the continuous embedding $B^\alpha_{p,s}(\mathbb{R}^n) \subset \mathcal{C}^{0,\alpha-\frac{n}{p}}(\mathbb{R}^n)$.    
\end{lem}

\section{An approximation lemma}\label{AppSec}
The main tool to prove Theorem \ref{mainthm} 
is the following approximation lemma, which allows us to approximate from below the function $f$ with a sequence of functions $(f_j)$ satisfying standard $p$-growth conditions and the same regularity properties of $f$ w.r.t.\ the $x$-variable (see \cite{grimaldi.ipocoana}).
\begin{lem}\label{apprlem1}
Let $f : \Omega \times \mathbb{R}^{N \times n} \rightarrow [0,+\infty), f = f(x, \xi)$, be a Carath\'{e}odory function satisfying the Uhlenbeck structure and assumptions \eqref{F1}, \eqref{F2}, \eqref{F3}, \eqref{F4}. Then, there exists
a sequence $(f_j)$ of Carath\'{e}odory functions $f_j: \Omega \times \mathbb{R}^{N \times n} \rightarrow [0,+\infty)$, convex with respect to the last variable, monotonically convergent to $f$, such that for a.e.\ $x,y \in \Omega$ and every $\xi \in \mathbb{R}^{N \times n}$
\begin{itemize}
\item[(i)] 
$f_j(x,\xi) = \tilde{f}_j (x,|\xi|)$;
\item[(ii)]
$f_j(x,\xi)\leq f_{j+1}(x, \xi) \leq f(x,\xi)$, $\forall j \in \mathbb{N}$;
\item[(iii)]
$\langle D_{\xi\xi} f_j(x,\xi)\lambda, \lambda\rangle \geq \nu_1(1+|\xi|^2)^\frac{p-2}{2}|\lambda|^2$,
$\forall \lambda \in \mathbb{R}^{N \times n}$,
with $\nu_1=\nu_1(\nu,p)$;
\item[(iv)] there exist $L_1$, independent of $j$, and $\bar{L}_1$, depending on $j$, such that
\begin{align*}
& 1/L_1(|\xi|^p-1) \leq f_j(x,\xi)\leq L_1(1 + |\xi|^2)^\frac{q}{2},\\
&f_j (x,\xi)\leq \bar{L}_1(j)(1 + |\xi|^2)^\frac{p}{2};
\end{align*} 
\item[(v)]there exists a constant $C(j) > 0$ such that
\begin{align*}
&|D_\xi f_j(x,\xi)-D_\xi f_j(y,\xi)| \leq |x-y|^{\gamma} {(k(x)+k(y))}(1 +|\xi|^2)^\frac{q-1}{2},\\
&|D_\xi f_j(x,\xi)-D_\xi f_j(y,\xi)| \leq  C(j)|x-y|^{\gamma} {(k(x)+k(y))}(1 +|\xi|^2)^\frac{p-1}{2}.
\end{align*}
\end{itemize}
\end{lem}
\noindent Moreover, supposing \eqref{eqf5} instead of \eqref{F4}, statement $(v)$ would change as follows:
\begin{itemize}
\item[(v)] there exists a constant $C(j) > 0$ such that
\begin{align*}
&|D_\xi f_j(x,\xi)-D_\xi f_j(y,\xi)| \leq |x-y|^{\gamma} (g_k(x)+g_k(y))(1 +|\xi|^2)^\frac{q-1}{2},\\
&|D_\xi f_j(x,\xi)-D_\xi f_j(y,\xi)| \leq  C(j)|x-y|^{\gamma} (g_k(x)+g_k(y))(1 +|\xi|^2)^\frac{p-1}{2}
\end{align*}
for a.e.\ $x,y \in \Omega$ such that $2^{-k} \text{diam}(\Omega) \leq |x-y| < 2^{-k+1}\text{diam}(\Omega)$ and for every $\xi \in \mathbb{R}^{N \times n}$.
\end{itemize}

\section{A priori estimate}\label{apriorisec}
We consider the following approximating problems
\begin{equation}\label{Papprox}
   \inf \left\{  \int_{\Omega}\left( f_j (x, Dw) - g \cdot w \right) dx \, :\,  w \in  W^{1,p}(\Omega, \mathbb{R}^N) \right\} ,
\end{equation}
where $(f_j)$ is the sequence of functions obtained applying Lemma \ref{apprlem1} to the integrand $f$ of the functional at \eqref{functional}. 
We denote by $u_j$ the minimum of the problem \eqref{Papprox}. Since $f_j$ satisfies standard $p$-growth, then $u_j$ solves the Euler-Lagrange system
\begin{equation}\label{System}
    \int_{\Omega} \langle D_{\xi}f_j (x, Du_{j}), D \varphi \rangle \, dx = \int_{\Omega} g \cdot \varphi \, dx, 
\end{equation}
for every $ \varphi \in W^{1,p}_{0}(\Omega, \mathbb{R}^N).$

The main aim of this section is to establish the following 
\begin{thm}\label{AppThm}
Let $g \in L^{m}_{loc}(\Omega, \mathbb{R}^N)$, with $p' \leq m \leq 2$. Let $u_j \in W^{1,p}(\Omega, \mathbb{R}^{N })$ be a local minimizer of \eqref{Papprox} under assumptions \eqref{F1}--\eqref{F4}, for some exponents $2\leq p  <\frac{n}{{\lambda}}< r$, where $\lambda:=  \min \{\gamma, \frac{m}{2} \}$, $p<q$ such that \eqref{gap.} is in force. 
If we a priori assume that $$V_p(Du_j) \in B^{{\lambda}}_{2,\infty,\text{loc}}(\Omega,\mathbb{R}^{N \times n}),$$ then the following estimates
\begin{align*}
    \left(  \int_{B_{R/4}}|Du_j|^\frac{np}{n-2 \beta} dx \right)^\frac{n-2 \beta}{n} & \leq c \left( \int_{B_{R}} (1 + |Du_j|^p) dx + \Vert k \Vert_{L^r (B_{R})} + \Vert g \Vert_{L^{m}(B_R)}  \right)^\sigma \qquad 0<\beta < \lambda
\end{align*}
and
\begin{align*}
   \int_{B_{R/4}} |\tau_h V_p (Du_j)|^2 dx \leq  c {|h|^{2 \lambda} }\left( \int_{B_{R}} (1 + |Du_j|^p) dx + \Vert k \Vert_{L^r (B_{R})} + \Vert g \Vert_{L^{m}(B_R)}  \right)^\sigma
\end{align*}
hold for every concentric balls $B_{R/4} \subset B_{R} \Subset \Omega$, where $c = c(n,p,q,m, \nu,l, L,R)$ and $\sigma= \sigma (n,p,q,m,\gamma)$ are positive constants independent of $j$.
\end{thm}

\begin{proof}
 In the following, we write $u$ instead of $u_j$ for simplicity of notation.
For further needs, we notice that from the hypothesis
 $$V_p(Du) \in B^\lambda_{2,\infty,{loc}}(\Omega,\mathbb{R}^{N \times n})$$
and Lemma \ref{besovembed}, we have
\begin{equation}\label{H-1}
    Du \in L^{\frac{np}{n-2 \beta}}_{loc}(\Omega ,\mathbb{R}^{N \times n}),
\end{equation}
for all $0< \beta < \lambda$. One can easily check that \eqref{gap.} yields
\begin{equation}\label{dis}
 \dfrac{r(2q-p)}{r-2} \leq \dfrac{np}{n-2 \beta} , 
\end{equation}
for every $\beta \in (\frac{\lambda n r}{nr+2(\lambda r -n)}, \lambda)$.
 
Fix a ball $B_R \Subset \Omega$ and consider radii $\frac{R}{4}<s<t<t'<R$, a cut-off function $\eta \in C_0^{\infty}(B_t)$, with $\eta=1$ on $B_{s}$, $0 \leq \eta \leq 1$, $|D \eta | \leq \frac{c}{t-s}$  and $|h|\leq \frac{t' - t}{2}$. We test \eqref{System} with the function
$$\varphi = \tau_{-h} \left(  \eta^2 \tau_h u \right) $$
thus obtaining 
\begin{equation*}
    \int_{\Omega} \langle D_{\xi}f_j (x, Du), \tau_{-h} D  \left(  \eta^2 \tau_h u \right) \rangle \, dx = \int_{\Omega} g \cdot \tau_{-h}  \left(  \eta^2 \tau_h u \right) \, dx, 
\end{equation*}
hence by $(ii)$ of Proposition \ref{rapportoincrementale}
\begin{equation}\label{IntEq}
    \int_{\Omega} \langle \tau_{h} D_{\xi}f_j (x, Du), D  \left(  \eta^2 \tau_h u \right) \rangle \, dx = \int_{\Omega}  g \cdot \tau_{-h} \left(  \eta^2 \tau_h u \right) \, dx =: B.
\end{equation}
Now we rewrite the integral in the left-hand side of the previous equality as follows
\begin{align}
    \int_{\Omega}& \langle D_{\xi}f_j (x+h, Du(x+h))- D_{\xi}f_j (x, Du(x)) , 2 \eta D \eta \tau_h u \rangle \, dx \notag \\
    &+ \int_{\Omega} \langle D_{\xi}f_j (x+h, Du(x+h))- D_{\xi}f_j (x, Du(x)) ,  \eta^2 \tau_h Du \rangle \, dx \notag \\
    &=  \int_{\Omega} \langle D_{\xi}f_j (x+h, Du(x+h))- D_{\xi}f_j (x+h, Du(x)) , 2 \eta D \eta \tau_h u \rangle \, dx \notag \\
    &+  \int_{\Omega} \langle D_{\xi}f_j (x+h, Du(x))- D_{\xi}f_j (x, Du(x)) , 2 \eta D \eta \tau_h u \rangle \, dx \notag \\
    &+  \int_{\Omega} \langle D_{\xi}f_j (x+h, Du(x+h))- D_{\xi}f_j (x+h, Du(x)) , \eta^2 \tau_h Du \rangle \, dx \notag \\
    &+  \int_{\Omega} \langle D_{\xi}f_j (x+h, Du(x))- D_{\xi}f_j (x, Du(x)) , \eta^2 \tau_h Du \rangle \, dx \notag \\
    & =: A_1 + A_2 + A_3 + A_4.
\end{align}
So we have from equation \eqref{IntEq}
\begin{equation}\label{SommaInt}
   A_3 \leq |A_1|  + |A_2| + |A_4| + |B| .
\end{equation}
We start by considering the term $|A_1|$. By using hypotesis \eqref{F3}, Young's and H\"older's equalities, we get
\begin{align}\label{A1}
    |A_1| &= \left| \int_{\Omega} \sum_{i, \alpha} \left( f_{\xi_{i}^{\alpha}}(x+h, Du(x+h))-f_{\xi_{i}^{\alpha}}(x+h, Du(x))\right) 2 \eta \, \eta_{x_i} \tau_h u^{\alpha}dx \right| \notag \\
    & \leq 2\int_{\Omega} \eta \, |D\eta| \left( 1 + |Du(x+h)|^2 + |Du(x)^2| \right)^{\frac{q-2}{2}} \, |\tau_h
Du| \, |\tau_h u| dx \notag \\
& \leq \varepsilon \int_{\Omega} \eta^2 \, \left( 1 + |Du(x+h)|^2 + |Du(x)^2| \right)^{\frac{p-2}{2}} \, |\tau_h
Du|^2 dx  \notag \\
& \qquad + c_\varepsilon \int_{\Omega}  |D \eta|^2 \left( 1 + |Du(x+h)|^2 + |Du(x)^2| \right)^{\frac{2q-p-2}{2}} \, |\tau_h u|^2 dx\notag \\
& \leq \varepsilon \int_{\Omega} \eta^2 \, \left( 1 + |Du(x+h)|^2 + |Du(x)^2| \right)^{\frac{p-2}{2}} \, |\tau_h
Du|^2 dx  \notag \\
& \qquad + \frac{c_\varepsilon}{(t-s)^2} \left( \int_{B_{t'}}(1+ |Du|^2)^{\frac{2q-p}{2}} dx \right)^{\frac{2q-p-2}{2q-p}} \, \left( \int_{B_{t}} |\tau_h u|^{2q-p} dx \right)^{\frac{2}{2q-p}},
\end{align}
where in the last inequality we used that $|D \eta| \leq \frac{c}{t-s}$. Next applying Lemma \ref{ldiff}, we have
\begin{align}
 |A_1| & \leq  \varepsilon \int_{\Omega} \eta^2 \, \left( 1 + |Du(x+h)|^2 + |Du(x)^2| \right)^{\frac{p-2}{2}} \, |\tau_h
Du|^2 dx   \notag \\
& \qquad + \frac{c_\varepsilon}{(t-s)^2} \, |h|^2 \, \left( \int_{B_{t'}}(1+ |Du|)^{2q-p} dx \right)^{\frac{2q-p-2}{2q-p}} \, \left( \int_{B_{t'}} |Du|^{2q-p}dx \right)^{\frac{2}{2q-p}} \notag \\
& \leq \varepsilon \int_{\Omega} \eta^2 \, \left( 1 + |Du(x+h)|^2 + |Du(x)^2| \right)^{\frac{p-2}{2}} \, |\tau_h
Du|^2  dx  \notag \\
& \qquad + \frac{c_\varepsilon}{(t-s)^2} \, |h|^2 \, \int_{B_{t'}}(1+ |Du|)^{2q-p} dx ,\notag 
\end{align}
where the last integral is finite by virtue of \eqref{H-1} and \eqref{dis}.

The ellipticity assumption \eqref{F2} implies
\begin{align}\label{A3}
    A_3 \geq L \int_{\Omega} \eta^2 \, \left( 1 + |Du(x+h)|^2 + |Du(x)^2| \right)^{\frac{p-2}{2}} \, |\tau_h
Du|^2 dx.
\end{align}
From condition \eqref{F4}, Lemma \ref{ldiff}, the properties of $\eta$ and H\"older's inequality, we derive
\begin{align}\label{A2}
    |A_2| &= \left| \int_{\Omega} \sum_{i, \alpha} \left( f_{\xi_{i}^{\alpha}}(x+h, Du(x))-f_{\xi_{i}^{\alpha}}(x, Du(x)) \right)2 \eta \, \eta_{x_i} \tau_h u^{\alpha} dx \right| \notag \\
    & \leq 2 \int_{\Omega} \eta |D \eta||\tau_h u|  |h|^{\gamma}(k(x+h)+ k(x))(1 + |Du|^2)^{\frac{q-1}{2}} dx \notag \\
    & \leq \frac{c}{t-s}|h|^{\gamma}  \int_{B_t}   |\tau_h u| (k(x+h)+ k(x))(1 + |Du|^2)^{\frac{q-1}{2}} dx \notag \\
      & \leq  \frac{c}{t-s} \, |h|^{\gamma} \, \left( \int_{B_{R}} k^{r} dx \right)^{\frac{1}{r}} \, \left( \int_{B_{t}}|\tau_h u|^{\frac{r-1}{r}} (1 +|Du|)^{\frac{r(q-1)}{r-1}} dx \right)^{\frac{r-1}{r}} \notag \\
        & \leq  \frac{c}{t-s} \, |h|^{\gamma} \, \left( \int_{B_{R}} k^{r} dx \right)^{\frac{1}{r}} \, \left( \int_{B_{t}}|\tau_h u|^{\frac{rq}{r-1}}dx\right)^{\frac{r-1}{rq}} \left( \int_{B_t} (1 +|Du|)^{\frac{r(q-1)}{r-1}} dx \right)^{\frac{r-1}{r}} \notag \\
        &\leq \frac{c}{t-s} \, |h|^{1+ \gamma} \, \left( \int_{B_{R}} k^{r} dx \right)^{\frac{1}{r}} \, \left( \int_{B_{t'}} (1 +|Du|)^{\frac{rq}{r-1}} dx \right)^{\frac{r-1}{r}},
\end{align}
where we also used that $Du \in L^{\frac{rq}{r-1}}_{loc}(\Omega)$ by virtue of \eqref{H-1} and \eqref{dis}.

Similarly as for the estimate of $A_2$, we obtain
\begin{align}\label{A4}
     |A_4| &= \left| \int_{\Omega} \sum_{i, \alpha} \left( f_{\xi_{i}^{\alpha}}(x+h, Du(x))-f_{\xi_{i}^{\alpha}}(x, Du(x)) \right) \eta^2 \, \tau_h u_{x_i}^{\alpha} dx \right| \notag \\
    & \leq c |h|^{\gamma}  \int_{\Omega} \eta^2 |\tau_h Du|  (k(x+h)+ k(x))( 1 + |Du|^2)^{\frac{q-1}{2}} dx \notag \\
    & \leq \varepsilon \int_{\Omega} \eta^2 \, \left( 1 + |Du(x+h)|^2 + |Du(x)^2| \right)^{\frac{p-2}{2}} \, |\tau_h Du|^2 dx \notag \\
    & \qquad + c_{\varepsilon} \int_{\Omega} \eta ^2   (k(x+h)+ k(x))^2 ( 1 + |Du|^2)^{\frac{2q-p}{2}} dx \notag \\
     & \leq \varepsilon \int_{\Omega} \eta^2 \, \left( 1 + |Du(x+h)|^2 + |Du(x)^2| \right)^{\frac{p-2}{2}} \, |\tau_h Du|^2 dx \notag \\
    & \qquad + c_{\varepsilon} |h|^{2 \gamma}   \left( \int_{B_{R}} k^{r} dx \right)^{\frac{2}{r}} \, \left( \int_{B_{t}} (1 +|Du|)^{\frac{r(2q-p)}{r-2}} dx \right)^{\frac{r-2}{r}} ,
\end{align}
where the last integral is finite by virtue of \eqref{H-1} and \eqref{dis}.

Now, we take care of $B$. Recalling that $2 \le m \le p'$, the H\"older's inequality and Lemma \ref{ldiff} give that
\begin{align}
    \left| B\right| 
    & \le \left( \int_{B_t} |g|^{m}dx \right)^\frac{1}{m} \left( \int_{B_t} |\tau_{-h}(\eta^2 \tau_h u)|^{m'}\right)^\frac{1}{m'} \notag\\
    &\leq c|h|\, \left( \int_{B_R} |g|^{m}dx \right)^\frac{1}{m}\,\left( \int_{B_{t}} |D(\eta^2 \tau_h u)|^{m'} dx \right)^{\frac{1}{m'}} \notag \\
    &\leq c|h|\, \left( \int_{B_R} |g|^{m}dx \right)^\frac{1}{m}\,\left( \int_{B_{t}} \eta^p |D \eta|^{m'} |\tau_h u|^{m'} dx \right)^{\frac{1}{m'}} \notag \\
    & \ \ \ \ +  c|h|\, \left( \int_{B_R} |g|^{m}dx \right)^\frac{1}{m}\,\left( \int_{B_{t}} \eta^{2m'}  |\tau_h Du|^{m'} dx \right)^{\frac{1}{m'}}. \notag 
\end{align}
Using again Lemma \ref{ldiff}, Young's inequality and that $|D \eta | \le \frac{c}{t-s}$, $0 \le \eta \le 1$ and  $2 \leq m' \leq p$, we have
\begin{align}\label{B}
    |B|  &\leq \dfrac{c}{t-s}|h|^2\, \left( \int_{B_R} |g|^{m}dx \right)^\frac{1}{m}\,\left( \int_{B_{R}}  |D u|^{m'} dx \right)^{\frac{1}{m'}} \notag \\
    & \ \ \ \ +  c_\varepsilon \ |h|^{m}\, \int_{B_R} |g|^{m}dx \notag\\
    & \ \ \ \ + \varepsilon\int_{\Omega} \eta^{2}  \left( 1+ |Du(x+h)|^2+|Du(x)|^2\right)^\frac{p-2}{2}|\tau_h Du|^2 dx .
\end{align}
Inserting \eqref{A1}, \eqref{A2}, \eqref{A3}, \eqref{A4} and \eqref{B} in \eqref{SommaInt}, we infer
\begin{align}
 L  & \int_{\Omega} \eta^2 \, \left( 1 + |Du(x+h)|^2 + |Du(x)^2| \right)^{\frac{p-2}{2}} \, |\tau_h
Du|^2 dx \notag \\
& \qquad \leq  3 \varepsilon \int_{\Omega} \eta^2 \, \left( 1  + |Du(x+h)|^2 + |Du(x)^2| \right)^{\frac{p-2}{2}} \, |\tau_h
Du|^2 dx \notag \\
& \qquad +\frac{c_\varepsilon}{(t-s)^2} |h|^2 \int_{B_{t'}}(1 + |Du|)^{2q-p} dx \notag \\
& \qquad + \frac{c}{t-s} \, |h|^{1+ \gamma} \, \left( \int_{B_{R}} k^{r} dx \right)^{\frac{1}{r}} \, \left( \int_{B_{t'}} (1 +|Du|)^{\frac{rq}{r-1}} dx \right)^{\frac{r-1}{r}} \notag \\
&  \qquad +c_\varepsilon \, |h|^{2 \gamma} \, \left( \int_{B_{R}} k^{r} dx \right)^{\frac{2}{r}} \, \left( \int_{B_{t'}} (1 +|Du|)^{\frac{r(2q-p)}{r-2}} dx \right)^{\frac{r-2}{r}} \notag \\
& \qquad + \dfrac{c}{t-s}|h|^2\, \left( \int_{B_R} |g|^{m}dx \right)^\frac{1}{m}\,\left( \int_{B_{R}}  |D u|^p dx \right)^{\frac{1}{p}} \notag \\
    & \qquad  +  c_\varepsilon \ |h|^{m}\, \int_{B_R} |g|^{m}dx. \notag
\end{align}
Choosing $\varepsilon = \frac{L}{6}$ and reabsorbing the first integral in the right-hand side by the left-hand side, we get
\begin{align}
    \int_{\Omega} \eta^2 \,&  \left( 1 + |Du(x+h)|^2 + |Du(x)^2| \right)^{\frac{p-2}{2}} \, |\tau_h Du|^2 dx \notag \\
&\leq \frac{c}{(t-s)^2} |h|^2 \int_{B_{t'}}(1 + |Du|)^{2q-p} dx \notag \\
& \, + \frac{c}{t-s} \, |h|^{1+ \gamma} \, \left( \int_{B_{R}} k^{r} dx \right)^{\frac{1}{r}} \, \left( \int_{B_{t'}} (1 +|Du|)^{\frac{rq}{r-1}} dx \right)^{\frac{r-1}{r}} \notag \\
&  \, +c \, |h|^{2 \gamma} \, \left( \int_{B_{R}} k^{r} dx \right)^{\frac{2}{r}} \, \left( \int_{B_{t'}} (1 +|Du|)^{\frac{r(2q-p)}{r-2}} dx \right)^{\frac{r-2}{r}} \notag \\
& \ + \dfrac{c}{t-s}|h|^2\, \left( \int_{B_R} |g|^{m}dx \right)^\frac{1}{m}\,\left( \int_{B_{R}}  |D u|^p dx \right)^{\frac{1}{p}} \notag \\
    & \ +  c \ |h|^{m}\, \int_{B_R} |g|^{m}dx .\notag
\end{align}
{Since $q >p$ and $r > 2$}, it holds
$$2q-p \leq \frac{r(2q-p)}{r-2} \quad \text{and} \quad \frac{rq}{r-1} \leq \frac{r(2q-p)}{r-2}.$$
Therefore, we have
\begin{align}
    \int_{\Omega} \eta^2 \, & \left( 1 + |Du(x+h)|^2 + |Du(x)^2| \right)^{\frac{p-2}{2}} \, |\tau_h Du|^2 dx \notag \\
&\leq \frac{c}{(t-s)^2} |h|^2 \left( \int_{B_{t'}}(1 + |Du|)^\frac{r(2q-p)}{r-2} dx  \right)^\frac{r-2}{r}\notag \\
& \, + \frac{c}{t-s} \, |h|^{1+ \gamma} \, \left( \int_{B_{R}} k^{r} dx \right)^{\frac{1}{r}} \, \left( \int_{B_{t'}} (1 +|Du|)^{\frac{r(2q-p)}{r-2}} dx \right)^{\frac{q(r-2)}{(2q-p)r}} \notag \\
&  \, +c \, |h|^{2 \gamma} \, \left( \int_{B_{R}} k^{r} dx \right)^{\frac{2}{r}} \, \left( \int_{B_{t'}} (1 +|Du|)^{\frac{r(2q-p)}{r-2}} dx \right)^{\frac{r-2}{r}} \notag \\
& \ + \dfrac{c}{t-s}|h|^2\, \left( \int_{B_R} |g|^{m}dx \right)^\frac{1}{m}\,\left( \int_{B_{R}}  |D u|^p dx \right)^{\frac{1}{p}} \notag \\
    & \  +  c \ |h|^{m}\, \int_{B_R} |g|^{m}dx .\label{Stima1}
\end{align}
We now consider the interpolation inequality
\begin{align}\label{interp1}
\left( \int_{B_{\rho}} |h|^\frac{r(2q-p)}{r-2}dx \right)^{\frac{r-2}{r(2q-p)}} \leq \left( \int_{B_\rho} |h|^p dx \right)^\frac{\pi}{p}  \, \, \left( \int_{B_\rho} |h|^\frac{np}{n-\beta p} dx \right)^{\frac{(n-2 \beta)(1-\pi)}{np}},
\end{align}
where $0< \pi <1$ is defined by
\begin{align}
\dfrac{r-2}{r(2q-p)}= \dfrac{\pi}{p}+ \dfrac{(1-\pi)(n-2 \beta)}{np} \notag
\end{align}
which implies
$$\pi= \dfrac{nr(p-q)-np+\beta r(2q-p)}{\beta r(2q-p)}, \ \ \ 1- \pi= \dfrac{n[r(q-p)+p]}{\beta r(2q-p)}.$$
Applying Lemma \ref{D1} in the left-hand side of \eqref{Stima1} and the interpolation inequality \eqref{interp1} in the right-hand side, we get

\begin{align}
       \int_{B_{t}}\eta ^2 |\tau_h V_p (Du)|^2 dx & \leq 
 \frac{c}{(t-s)^2} |h|^2 \left( \int_{B_{t'}}(1 + |Du|)^p dx \right)^{\frac{(2q-p) \pi}{p}} \, \left( \int_{B_{t'}}(1 + |Du|)^{\frac{np}{n-2 \beta}} dx \right)^{\frac{n-2 \beta}{np}(2q-p)(1- \pi)}  \notag \\
& \, +  \frac{c}{t-s} \, |h|^{1+ \gamma} \, \left( \int_{B_{R}} k^{r} dx \right)^{\frac{1}{r}} \, \left( \int_{B_{t'}} (1 +|Du|)^p dx \right)^{\frac{\pi q}{p}} \cdot \notag \\
& \qquad \qquad \cdot \left( \int_{B_{t'}}(1 + |Du|)^{\frac{np}{n-2 \beta}} dx \right)^{\frac{n-2 \beta}{np}q(1- \pi)} \notag \\
& \, + c \, |h|^{2 \gamma} \, \left( \int_{B_{R}} k^{r} dx \right)^{\frac{2}{r}} \, \left( \int_{B_{t'}} (1 +|Du|)^p dx \right)^{\frac{\pi(2 q-p)}{p}} \cdot \notag \\
& \qquad \qquad \cdot \left( \int_{B_{t'}}(1 + |Du|)^{\frac{np}{n-2 \beta}} dx \right)^{\frac{n-2 \beta}{np}(2 q- p)} \notag \\
& \ + \dfrac{c}{t-s}|h|^2\, \left( \int_{B_R} |g|^{m}dx \right)^\frac{1}{m}\,\left( \int_{B_{R}}  |D u|^p dx \right)^{\frac{1}{p}} \notag \\
    & \  +  c \ |h|^{m}\, \int_{B_R} |g|^{m}dx \notag  .
\end{align}
Assuming $\tilde{q}$ as the conjugate of $ \frac{p}{q(1-\pi)}$, $\tilde{p}$ as the conjugate of $\frac{p}{(2q-p)(1- \pi)}$ and applying Young's inequality
\begin{align}
     \int_{B_{t}} |\tau_h V_p (Du)|^2 dx & \leq \, \frac{\varepsilon}{3}|h|^2 \left(  \int_{B_{t'}}(1+ |Du|)^{\frac{np}{n-2 \beta}} dx \right)^{\frac{n-2 \beta}{n}} \notag \\
     & \qquad +   \frac{c_\varepsilon}{(t-s)^{2\tilde{p}}}|h|^2 \left(  \int_{B_{2R}}(1+ |Du|)^p dx\right)^{\frac{(2q-p)\pi}{p}\tilde{p}} \notag \\
     & \qquad +  \frac{\varepsilon}{3}|h|^{1+ \gamma} \left(  \int_{B_{t'}}(1+ |Du|)^\frac{np}{n-2 \beta} dx\right)^\frac{n-2 \beta}{n} \notag \\
    & \qquad +  \frac{c_\varepsilon}{R^{\tilde{q}}}|h|^{1 + \gamma}\left(  \int_{B_{R}} k^r dx \right)^{\frac{ \tilde{q}}{r}} \left(  \int_{B_{t'}}(1+ |Du|)^p dx \right)^{\frac{q \pi \tilde{q}}{p}}  \notag \\
     & \qquad +   \frac{\varepsilon}{3}|h|^{2\gamma} \left(  \int_{B_{t'}}(1+ |Du|)^\frac{np}{n-2 \beta} dx \right)^\frac{n-2 \beta}{n} \notag \\
    & \qquad + c_\varepsilon|h|^{2\gamma}\left(  \int_{B_{R}} k^r dx \right)^{\frac{ 2\tilde{q}}{r}} \left(  \int_{B_{t'}}(1+ |Du|)^p dx \right)^{\frac{(2q-p) \pi \tilde{q}}{p}}  \notag \\
    & \qquad + \dfrac{c}{t-s}|h|^2\, \left( \int_{B_R} |g|^{m}dx \right)^\frac{1}{m}\,\left( \int_{B_{R}}  |D u|^p dx \right)^{\frac{1}{p}} \notag \\
    & \qquad  +  c \ |h|^{m}\, \int_{B_R} |g|^{m}dx \notag \\
     & \leq |h|^{2 \lambda} \Biggl\{ \varepsilon \left(  \int_{B_{t'}}(1+ |Du|)^\frac{np}{n-2 \beta} dx \right)^\frac{n-2 \beta}{n} \notag +  \frac{c_\varepsilon}{R^{2\tilde{p}}} \left(  \int_{B_{t'}}(1+ |Du|)^p dx\right)^{\frac{(2q-p)\pi}{p}\tilde{p}} \notag \\
     & \qquad + \frac{c_\varepsilon}{(t-s)^{\tilde{q}}}\left(  \int_{B_{R}} k^r dx \right)^{\frac{ \tilde{q}}{r}} \left(  \int_{B_{t'}}(1+ |Du|)^p dx \right)^{\frac{q \pi \tilde{q}}{p}}  \notag \\
     & \qquad +   c_\varepsilon \left(  \int_{B_{R}} k^r dx \right)^{\frac{ 2 \tilde{q}}{r}} \left(  \int_{B_{t'}}(1+ |Du|)^p dx \right)^{\frac{(2q-p) \pi \tilde{q}}{p}}  \notag \\
     & \qquad  + \dfrac{c}{t-s}\, \left( \int_{B_R} |g|^{m}dx \right)^\frac{1}{m}\,\left( \int_{B_{R}}  |D u|^p dx \right)^{\frac{1}{p}} \notag \\
    & \qquad  +  c \, \int_{B_R} |g|^{m}dx \Biggr\}.\notag
    \end{align}
    where in the last inequality we used the fact that $|h| \leq 1$ and $\lambda:=  \min \{\gamma, \frac{m}{2} \}$.\\
 From Lemma \ref{EmbendMigliore} it follows that
 \begin{align}
\left(  \int_{B_{s}}|Du|^\frac{np}{n-2 \beta} dx \right)^\frac{n-2 \beta}{n} & \leq \varepsilon \left(  \int_{B_{t'}}(1+ |Du|)^\frac{np}{n-2 \beta} dx \right)^\frac{n-2 \beta}{n} \notag +  \frac{c_\varepsilon}{(t-s)^{2\tilde{p}}} \left(  \int_{B_{t'}}(1+ |Du|)^p dx\right)^{\frac{(2q-p)\pi}{p} \tilde{p}} \notag \\
     & \qquad + \frac{c_\varepsilon}{(t-s)^{\tilde{q}}}\left(  \int_{B_{R}} k^r dx \right)^{\frac{ \tilde{q}}{r}} \left(  \int_{B_{t'}}(1+ |Du|)^p dx \right)^{\frac{q \pi \tilde{q}}{p}}  \notag \\
     & \qquad +   c_\varepsilon \left(  \int_{B_{R}} k^r dx \right)^{\frac{ 2 \tilde{q}}{r}} \left(  \int_{B_{t'}}(1+ |Du|)^p dx \right)^{\frac{(2q-p) \pi \tilde{q}}{p}}  \notag \\
     & \qquad  + \dfrac{c}{t-s}\, \left( \int_{B_R} |g|^{p'}dx \right)^\frac{1}{p'}\,\left( \int_{B_{R}}  |D u|^p dx \right)^{\frac{1}{p}} \notag \\
    & \qquad  +  c \ \int_{B_R} |g|^{m}dx. \notag
 \end{align}
Applying Lemma \ref{lm2}
\begin{align*}
    \left(  \int_{B_{R/4}}|Du|^\frac{np}{n-2 \beta} dx \right)^\frac{n-2 \beta}{n} & \leq c \left( \int_{B_{R   }} (1 + |Du|^p) dx + \Vert k \Vert_{L^r (B_{R})} + \Vert g \Vert_{L^{m}(B_R)}\right)^\sigma 
\end{align*}
which is finite since $u \in W^{1,p}_{loc}(\Omega, \mathbb{R}^N)$, $k \in L^r(\Omega)$ and $g \in L^{m}_{loc}(\Omega, \mathbb{R}^N)$ and so
\begin{align}
   \int_{B_{R/4}} |\tau_h V_p (Du)|^2 dx \leq  c |h|^{2 \lambda} \left( \int_{B_{R}} (1 + |Du|^p) dx + \Vert k \Vert_{L^r (B_{R})} +  \Vert g \Vert_{L^{m}(B_R)} \right)^\sigma, \label{StimaTeo1}
\end{align}
for positive constants $c = c(n,p,q,m, \nu,l, L,R)$ and $\sigma= \sigma (n,p,q,m,\gamma)$ independent of $j$.
\end{proof}
\section{Proof of Theorem \ref{mainthm}}\label{mainthmsec}
In the case of $p$-growth conditions, i.e.\ when $p=q$, we have the following higher differentiability result.
\begin{thm}\label{ReThm}
Let  $g \in L^{m}_{loc}(\Omega, \mathbb{R}^N)$, with $p' \leq m \leq 2$. Moreover, assume that $f$ satisfies \eqref{F1}--\eqref{F4} for exponents $2\leq p =q <\frac{n}{{\lambda}}$,
where $\lambda:=  \min \{\gamma, \frac{m}{2} \}$, and for a function $k \in L^\infty_{loc}(\Omega)$.  
If $u\in W^{1,p}_{loc}(\Omega, \mathbb{R}^N)$ is a local minimizer of \eqref{functional}, then $V_p(Du)  \in B^\lambda_{2,\infty,\textrm{\text{loc}}}(\Omega, \mathbb{R}^{N \times n})$.   
\end{thm}
\begin{proof}
The proof goes exactly as the one of Theorem \ref{AppThm}, but now the stronger assumptions on the data, i.e.\ $p=q$ and $k \in L^\infty_{loc}(\Omega)$, allow us to derive the a priori estimate \eqref{StimaTeo1} for $V_p(Du)$ without assuming any higher integrability property on the gradient.
\end{proof}

For a given ball $B_R \Subset \Omega$, we consider the variational problems
\begin{equation}\label{Pfj}
    \inf \left\{  \int_{B_R} [f_j(x,Dv)-g \cdot v]  dx \ : \ v \in W^{1,p}_0(B_R, \mathbb{R}^N)+u \right\}, \quad j \in \mathbb{N},
\end{equation}
where $(f_j)$ is the sequence of functions defined in Lemma \ref{apprlem1}. The lower semicontinuity and the $p$-growth conditions of the integrand $f_j$ give the existence of a unique solution $u_j \in W^{1,p}_0(B_R, \mathbb{R}^N)+ u $ of the problem \eqref{Pfj}. 
\\
Now, fix a non-negative smooth kernel $\phi \in \mathcal{C}^\infty_0(B_1(0))$ such that $\int_{B_1(0)} \phi =1$ and consider the corresponding family of mollifiers $(\phi_\varepsilon)_{\varepsilon >0}$. We set
$$k_\varepsilon=k * \phi_\varepsilon$$
and
\begin{equation}
f_{j}^\varepsilon(x,\xi)= \int_{B_1(0)} \phi(y)  f_j(x+\varepsilon y, \xi) dy \label{f_j}
\end{equation}
on $B_R$, for every $\varepsilon < \text{dist}(B_R,\Omega)$. One can easily check that $f_j^\varepsilon$ satisfies assumptions \eqref{F1}--\eqref{F3} and the set of conditions:
\begin{equation}\label{F1*}
    1/L_1(|\xi|^p-1) \leq f_{j}^\varepsilon(x,\xi) \leq L_1(1+|\xi|^p), \tag{F1*}
\end{equation}
\begin{equation}\label{F2*}
    \langle D_\xi f_{j}^\varepsilon(x, \xi)-D_\xi f_{j}^\varepsilon(x, \eta),\xi-\eta\rangle  \ge \nu_1 (1 +|\xi|^2+|\eta|^2)^\frac{p-2}{2}|\xi - \eta|^2, \tag{F2*}
\end{equation}
\begin{equation}\label{F3*}
    |D_\xi f_{j}^\varepsilon(x, \xi)-D_\xi f_{j}^\varepsilon(x, \eta)| \le l(j) (1 +|\xi|^2+|\eta|^2)^\frac{p-2}{2}|\xi - \eta|. \tag{F3*}
\end{equation}
By virtue of assumption \eqref{F4}, we have that for a.e. $x,y \in B_R$ and every $\xi \in \mathbb{R}^{N \times n}$,
\begin{equation}\label{F4*}
    |D_\xi f_{j}^\varepsilon(x,\xi)-D_\xi f_{j}^\varepsilon(y, \xi)| \leq |x-y|^{\gamma} (k_\varepsilon(x)+k_\varepsilon(y)) (1 +|\xi|^2)^\frac{q-1}{2} , \tag{F4*}
\end{equation}
\begin{equation}\label{F4**}
    |D_\xi f_{j}^\varepsilon(x,\xi)-D_\xi f_{j}^\varepsilon(y, \xi)| \leq \tilde{c}(j) |x-y|^{\gamma} (k_\varepsilon(x)+k_\varepsilon(y)) (1 +|\xi|^2)^\frac{p-1}{2}. \tag{F4**}
\end{equation}
Let us define the boundary value problem in $B_R \Subset \Omega$
\begin{equation}\label{Pfj1}
    \inf \left\{  \int_{B_R} [f_j^\varepsilon(x,Dv)-g \cdot v]  dx \ : \ v \in W^{1,p}_0(B_R, \mathbb{R}^N)+u_j \right\}.
\end{equation}
By the strict convexity of $f_j^\varepsilon$, there exists a unique minimum $u_j^\varepsilon \in W^{1,p}_0(B_R, \mathbb{R}^N)+ u_j $ of the problem \eqref{Pfj1}. The $p$-growth from above in \eqref{F1*} implies that $u_j^\varepsilon$ satisfies the Euler-Lagrange system
\begin{equation}\label{EL}
    \int_{B_R} \langle D_\xi f_j^\varepsilon(x,Du_j^\varepsilon),D \varphi \rangle  dx = \int_{B_R} g \cdot \varphi  dx, \quad \text{for every } \varphi \in W^{1,p}_0(B_R, \mathbb{R}^N).
\end{equation}
 Exploiting the ellipticity assumption \eqref{F2*}, choosing $\varphi= u_j^\varepsilon -u_j$ as test function in \eqref{EL} and \eqref{Pfj}, we have
\begin{align}
    \nu_1 & \int_{B_R} (1+|Du_j|^2+|Du_{j}^\varepsilon|^2)^{\frac{p-2}{2}}
    |Du_{j}^\varepsilon-Du_j|^2 dx \notag\\
    & \le \int_{B_R} \langle D_\xi f_j^\varepsilon(x,Du_j^\varepsilon) -  D_\xi f_j^\varepsilon(x,Du_j),D (u_j^\varepsilon - u_j) \rangle  dx
 \notag\\
    & = \underbrace{ \int_{B_R} \langle D_\xi f_j^\varepsilon(x,Du_j^\varepsilon) ,D (u_j^\varepsilon - u_j) \rangle  dx}_{= \ \int_{B_R} g \cdot (u_j^\varepsilon -u_j)dx} \notag\\
    & \ \ \ \ - \int_{B_R} \langle D_\xi f_j^\varepsilon(x,Du_j),D (u_j^\varepsilon - u_j) \rangle  dx \notag\\
    & \ \ \ \ + \int_{B_R} \langle D_\xi f_j(x,Du_j),D (u_j^\varepsilon - u_j) \rangle  dx \notag\\
    & \ \ \ \ -\underbrace{ \int_{B_R} \langle D_\xi f_j(x,Du_j),D (u_j^\varepsilon - u_j) \rangle  dx}_{= \ \int_{B_R} g \cdot (u_j^\varepsilon -u_j)dx} \notag
\end{align}
and so
\begin{align*}
     & \int_{B_R} (1+|Du_j|^2+|Du_{j}^\varepsilon|^2)^{\frac{p-2}{2}}
    |Du_{j}^\varepsilon-Du_j|^2 dx \\
    & \le c \int_{B_R} \langle D_\xi f_j(x,Du_j) - D_\xi f_j^\varepsilon(x,Du_j),D (u_j^\varepsilon - u_j) \rangle  dx,
\end{align*}
for a constant $c:=c(\nu,p)$. Applying H\"older's inequality in the right hand side of the previous estimate, we infer
\begin{align}
   & \int_{B_R} 
    |Du_{j}^\varepsilon-Du_j|^p dx \notag\\
   & \le c \int_{B_R} (1+|Du_j|^2+|Du_{j}^\varepsilon|^2)^{\frac{p-2}{2}}
    |Du_{j}^\varepsilon-Du_j|^2 dx \notag\\
    & \le c\left( \int_{B_R} | D_\xi f_j(x,Du_j) - D_\xi f_j^\varepsilon(x,Du_j)|^\frac{p}{p-1} dx  \right)^\frac{p-1}{p} \left( \int_{B_R} |D (u_j^\varepsilon - u_j) |^p  dx \right)^\frac{1}{p}, \notag
\end{align}
where the first inequality follows from the fact that $p \ge 2$. 
Next, we divide both sides by $$ \left( \int_{B_R} |D (u_j^\varepsilon - u_j) |^p  dx \right)^\frac{1}{p}$$ thus getting
\begin{equation}\label{SCov}
   \int_{B_R} 
    |Du_{j}^\varepsilon-Du_j|^p dx \le c   \int_{B_R} | D_\xi f_j(x,Du_j) - D_\xi f_j^\varepsilon(x,Du_j)|^\frac{p}{p-1} dx  .
\end{equation}
Assumptions \eqref{F1*} and \eqref{F2*} yields
$$|D_\xi f_j^\varepsilon (x,Du_j)| \le c(j) (1+|Du_j|^{p-1}) \quad \text{a.e.\ in } B_R$$
with $c(j)$ independent of $\varepsilon$. Moreover, $|Du_j|^{p-1} \in L^{\frac{p}{p-1}}(B_R)$ and,   by the very definition of $f_j^\varepsilon$, we have $$D_\xi f_j^\varepsilon (x,Du_j) \to D_\xi f_j (x,Du_j) \quad \text{a.e.\ in } B_R \text{ \ \ \ as } \varepsilon \to 0^+$$ and so the Dominated convergence theorem gives that
$$\lim_{\varepsilon \to 0^+} \int_{B_R} | D_\xi f_j(x,Du_j) - D_\xi f_j^\varepsilon(x,Du_j)|^\frac{p}{p-1} dx =0. $$
Therefore, passing to the limit as $\varepsilon \to 0^+$ in \eqref{SCov}, 
\begin{equation}\label{strongconv}
    u_j^\varepsilon \to u_j \text{ \ strongly in \ } W^{1,p}(B_R, \mathbb{R}^N).
\end{equation} 
By the regularity Theorem \ref{ReThm}, we have that $V_p(Du_j^\varepsilon)  \in B^\lambda_{2,\infty,{{loc}}}(B_R, \mathbb{R}^{N \times n})$. Then, $u_j^\varepsilon$ satisfies the a priori estimate
\begin{align}
    \int_{B_\rho} |Du_j^\varepsilon|^\frac{np}{n-2 \beta} dx \le C \left( \int_{B_R}(1+|Du_j^\varepsilon|^p) dx+ \Vert k_\varepsilon \Vert_{L^r(B_R)}+ \Vert g \Vert_{L^{m}(B_R)}\right)^\sigma \label{APS}
\end{align}
for every $\beta < \lambda$ and all concentric balls $B_\rho \Subset B_R$, with constants $C:=C(n,p,q,m,\nu,L,l,R,\rho)$ and $\sigma:=\sigma(n,p,q,m,\gamma)$, both independent of $j$ and $\varepsilon$. Moreover, the right-hand side of inequality \eqref{APS} is uniformly bounded with respect to $\varepsilon$ by virtue of \eqref{strongconv} and the strong convergence of $k_\varepsilon$ to $k$ in $L^r_{loc}(B_R)$. Hence, by the weak lower semicontinuity, \eqref{strongconv} and the strong convergence of $k_\varepsilon$ to $k$ in $L^r_{loc}(B_R)$, we obtain
\begin{align}
    \int_{B_\rho} |Du_j|^\frac{np}{n-2 \beta} dx \le & \liminf_{\varepsilon}  \int_{B_\rho} |Du_j^\varepsilon|^\frac{np}{n-2 \beta} dx \notag\\
    \le & \ C \left( \int_{B_R}(1+|Du_j|^p) dx+ \Vert k \Vert_{L^r(B_R)}+ \Vert g \Vert_{L^{m}(B_R)} \right)^\sigma. \label{apriori}
\end{align}
Now, by the growth assumption at (iv) of Lemma \ref{apprlem1}, the minimality of $u_j$ and using $u$ as test function, we get
\begin{align}
    \int_{B_R} (|Du_j|^p-1)dx \le & \
    L_1 \int_{B_R} f_j(x,Du_j)dx \notag \\
     \le  & \ L_1 \int_{B_R} [f_j(x,Du_j)-g \cdot u_j+ g \cdot u_j]dx \notag\\
    \le & \ L_1 \int_{B_R} [f_j(x,Du)-g \cdot u]dx +L_1 \int_{B_R} g \cdot u_jdx \notag\\
    \le & \ L_1  \int_{B_R} [f(x,Du)-g \cdot u]dx + L_1\int_{B_R} g \cdot u_jdx, \label{est}
\end{align}
where in the last inequality we used the monotonicity of $(f_j)$. From Sobolev-Poicaré inequality, we derive that
\begin{align}
    \Vert u_j -u\Vert_{L^p(B_R)} \leq c \Vert Du_j - Du \Vert_{L^p(B_R)}, \notag
\end{align}
that implies 
\begin{align}
  \Vert u_j \Vert_{L^p(B_R)}    \le c(n,p,R) \left(  \Vert Du_j \Vert_{L^p(B_R)}+\Vert u \Vert_{W^{1,p}(B_R)}\right) \notag
\end{align}
and so the second integral in the right hand side of \eqref{est} can be estimated as follows
\begin{align}\label{piccione}
    \int_{B_R} g \cdot u_jdx \le & \ \Vert g \Vert_{L^{p'}(B_R)} \Vert u_j \Vert_{L^p(B_R)} \notag\\
    \le & \ c(n,p,R) \Vert g \Vert_{L^{p'}(B_R)} \left(  \Vert Du_j \Vert_{L^p(B_R)}+\Vert u \Vert_{W^{1,p}(B_R)}\right) \notag\\
    \le & \ \dfrac{1}{2L_1} \Vert Du_j \Vert_{L^p(B_R)}^p+ c(n,p,L_1,R) \left( \Vert g \Vert_{L^{p'}(B_R)}^{p'}+ \Vert u \Vert_{W^{1,p}(B_R)}^p  \right), 
\end{align}
where we also used H\"older's and Young's inequalities. Putting the previous estimate in \eqref{est} and reabsorbing the term $\int_{B_R} |Du_j|^p dx$ from the right hand side by the left hand side, we get
\begin{equation}
    \int_{B_R} |Du_j|^pdx \le c  \int_{B_R} [f(x,Du)-g \cdot u]dx + c  \left( 1+\Vert g \Vert_{L^{p'}(B_R)}^{p'}+ \Vert u \Vert_{W^{1,p}(B_R)}^p  \right), \label{stima}
\end{equation}
for a positive constant $c$ independent of $j$.
Therefore, up to subsequences, $(u_j)$ converges weakly in $W^{1,p}(B_R, \mathbb{R}^N)$ as $j \to \infty$ to a function $v \in W^{1,p}_0(B_R, \mathbb{R}^N)+u$ and by estimates \eqref{apriori} and \eqref{stima} we have
\begin{align}
    \int_{B_\rho} |Dv|^\frac{np}{n-2 \beta} dx \le & \liminf_{j}  \int_{B_\rho} |Du_j|^\frac{np}{n-2 \beta} dx \notag\\
    \le & \ C \left( \int_{B_R}(1+f(x,Du)-g \cdot u) dx+ \Vert u \Vert_{W^{1,p}(B_R)}+ \Vert k \Vert_{L^r(B_R)}+ \Vert g \Vert_{L^{m}(B_R)} \right)^\sigma.
\end{align}
The convergence of $(f_j)$ to $f$, the weak lower semicontinuity of $f_j$ and the minimality of $u_j$ give that
\begin{align*}
    \int_{B_R} [f(x,Dv)-g \cdot v]dx \le & \liminf_{j_0} \int_{B_R} [f_{j_0}(x,Dv)-g \cdot v]dx \notag\\
    \le & \liminf_{j_0} \liminf_j\int_{B_R} [f_{j_0}(x,Du_j)-g \cdot u_j]dx \notag\\
    \le & \liminf_j\int_{B_R} [f_{j}(x,Du_j)-g \cdot u_j]dx \notag\\
    \le & \liminf_j\int_{B_R} [f_{j}(x,Du)-g \cdot u]dx \notag\\
    \le & \int_{B_R} [f(x,Du)-g \cdot u]dx,
\end{align*}
which implies $v=u$ a.e.\ in $B_R$ for the strict convexity of $f$. Then, $Du \in L^\frac{np}{n-2 \beta}_{loc}(B_R, \mathbb{R}^{N \times n})$ and the estimate \eqref{StimaTeo1} holds.
\section{Widely degenerate case}\label{WDSec}
In this section, we give the proof of Theorem \ref{WDT}. We argue as in \cite{Cupini1,Cupini2}.

\begin{proof}[Proof of Theorem \ref{WDT}] Let us consider the integrands
$$ f_j (x, \xi)= f(x, \xi) + \frac{1}{j}(1+ \lvert \xi \rvert^p), \qquad \forall j \in \mathbb{N}.$$
For a given ball $B_R \Subset \Omega$, we consider the variational problems
\begin{equation}\label{Wdfj}
    \inf \left\{  \int_{B_R} [f_j(x,Dv) +  \text{arctan}(|v-u|^2)  -g \cdot v]  dx \ : \ v \in W^{1,p}_0(B_R, \mathbb{R}^N)+u \right\},
\end{equation}
The lower semicontinuity and the $p$-growth conditions of the integrand $f_j$ give the existence of a unique solution $u_j \in W^{1,p}_0(B_R, \mathbb{R}^N)+ u $ of the problem \eqref{Wdfj}. It is easy to show that $f_j$ satisfies \eqref{F1}--\eqref{F4} for every $ j \in \mathbb{N}$. In presence of the perturbation term $\int_{B_R} \text{arctan}(|v-u|^2)  dx $, the proof of Theorem \ref{mainthm} can be easily adapted because of the boundedness of its integral and of its derivative with respect to the variable $v$. 
Hence we have $ V_p(Du_j) \in B^{\lambda}_{2, \infty, loc}(B_R)$ and 
\begin{align}\label{720}
    \left(  \int_{B_{R/4}}|Du_j|^\frac{np}{n-2 \beta} dx \right)^\frac{n-2 \beta}{n} & \leq c \left( \int_{B_{R}} (1 + |Du_j|^p) dx + \Vert k \Vert_{L^r (B_{R})} + \Vert g \Vert_{L^{m}(B_R)}  \right)^\sigma \qquad 0<\beta < \lambda
\end{align}
where $c$ and $\sigma$ are positive constants independent of $j$.\\
We observe that from the definition of $f_j$ and the assumption \eqref{H1}, it holds
\begin{align}\label{xi*}
\lvert  \xi \rvert^p \leq c(1 + f_j(x, \xi)), \qquad \forall \xi \in \mathbb{R}^{N \times n}.
\end{align}
By the minimality of $u_j$ and using $u$ as test function, we infer
\begin{align}\label{signore}
    \int_{B_R} \left[ f_j (x, D u_j) +\text{arctan}(|u_j-u|^2)  - g \cdot u_j \right] dx \leq  \int_{B_R} \left[f_j (x, D u) - g \cdot u \right]  dx,
\end{align}
so  using \eqref{xi*} in the left-hand side of \eqref{signore} we get
\begin{align}\label{liliana}
    \int_{B_R} \lvert D u_j\rvert^p dx &\leq c \int_{B_R} (1 + f_j (x, D u_j)) dx  \notag \\
    & \leq  c\int_{B_R} \left[ 1+ f_j (x, D u_j)  + \text{arctan}(|u_j-u|^2)  - g \cdot u_j + g \cdot u_j\right] dx \notag \\
     & \leq  c\int_{B_R} \left[ 1+ f_j (x, D u) - g \cdot u + g \cdot u_j\right] dx \notag \\
     & \leq  c\int_{B_R} \left[ 1+ f (x, D u)+ \lvert Du\rvert^p - g \cdot u + g \cdot u_j\right] dx. 
\end{align}
We recall the estimate \eqref{piccione}
\begin{align}
    \int_{B_R} g \cdot u_jdx \le \ \dfrac{1}{2} \Vert Du_j \Vert_{L^p(B_R)}^p+ c(n,p,R) \left( \Vert g \Vert_{L^{p'}(B_R)}^{p'}+ \Vert u \Vert_{W^{1,p}(B_R)}^p  \right). \notag
\end{align}
Inserting the previous estimate in \eqref{liliana} and reabsorbing the term $\Vert Du_j \Vert_{L^p(B_R)}^p$ from the right hand side in the left hand side, we obtain
\begin{align}\label{lieve}
    \int_{B_R} \lvert D u_j\rvert^p dx &\leq   c\int_{B_R} \left[ 1+ f (x, D u)+ \lvert Du\rvert^p - g \cdot u + \lvert g \rvert^{p'}+ \lvert u \rvert^{p} \right] dx \notag\\
    & \le c\int_{B_R} \left[ 1+ f (x, D u)+ \lvert Du\rvert^p + \lvert g \rvert^{p'}+ \lvert u \rvert^{p} \right] dx,  
\end{align}
where in the last line we used Young's inequality. 
Hence $(Du_j)_j$ is bounded in $L^p(B_R)$ and so there exists a function $v \in W^{1,p}(B_R)$ such that $ u_j  \rightharpoonup v$ weakly.\\
Inserting \eqref{lieve} in \eqref{720}
\begin{align}
    \left(  \int_{B_{R/4}}|Du_j|^\frac{np}{n-2 \beta} dx \right)^\frac{n-2 \beta}{n} & \leq c \left( \int_{B_{R}} (1 +f (x, D u)) dx + \Vert k \Vert_{L^r (B_{R})} + \Vert g \Vert_{L^{m}(B_R)}  +\Vert u \Vert_{W^{1,p}(B_R)}^p\right)^\sigma. \notag
\end{align}
From weak lower semicontinuity, passing to the limit as $j \to + \infty$  we derive
\begin{align}
\left(  \int_{B_{R/4}}|Dv|^\frac{np}{n-2 \beta} dx \right)^\frac{n-2 \beta}{n} &\leq  \liminf_j \left(  \int_{B_{R/4}}|Du_j|^\frac{np}{n-2 \beta} dx \right)^\frac{n-2 \beta}{n} \notag \\
   & \leq c \left( \int_{B_{R}} (1 +f (x, D u)) dx + \Vert k \Vert_{L^r (B_{R})} + \Vert g \Vert_{L^{m}(B_R)}  +\Vert u \Vert_{W^{1,p}(B_R)}^p\right)^\sigma. \label{gigi}
\end{align}
On the other hand $u_j$ also satisfies \eqref{StimaTeo1} and since estimate \eqref{lieve} holds, we have
\begin{align}
   \int_{B_{R/4}} |\tau_h V_p (Du_j)|^2 dx \leq  c |h|^{2 \lambda}  \left( \int_{B_{R}} (1 +f (x, D u))dx + \Vert k \Vert_{L^r (B_{R})} + \Vert g \Vert_{L^{m}(B_R)}  +\Vert u \Vert_{W^{1,p}(B_R)}^p\right)^\sigma. \notag
\end{align}
The sequence $(V_p (Du_j))_j$ is bounded in $B_{2, \infty}^{\lambda}(B_{R/4})$ and therefore is bounded in $B_{2, \infty,loc}^{\lambda}(B_{R})$ thanks to a standard covering argument. {Besov spaces can be seen as interpolation spaces between $L^p$ and $W^{1,p}$ (see \cite[Chapter 6]{bergh}), so applying classical interpolation theory (\cite[Theorem 3.8.1]{bergh})}, there exists a function $w$ such that $$V_p (Du_j) \rightharpoonup w \qquad \text{ weakly in } 
 B_{2, \infty,loc}^{\lambda}(B_{R}).$$ Therefore $V_p (Du_j)$ converges  strongly to $w$ in $L^2_{loc}(B_R)$ and recalling that $$|V_p (Du_j)|^2 = |Du_j|^p$$ and 
 $$Du_j\rightharpoonup Dv \qquad \text{ weakly in } 
 L^p(B_R),$$ we get, by uniqueness of weak limit that  $w= V_p(Dv).$
 Using weak lower semicontinuity we infer
 \begin{align}
   \int_{B_{R/4}} |\tau_h V_p (Dv)|^2 dx &\leq \liminf_j \int_{B_{R/4}} |\tau_h V_p (Du_j)|^2 dx   \notag \\
   &\leq c |h|^{2 \lambda}  \Big( \int_{B_{R}} (1 +f (x, D u)) dx + \Vert k \Vert_{L^r (B_{R})}+ \Vert g \Vert_{L^{m}(B_R)}  +\Vert u \Vert_{W^{1,p}(B_R)}^p\Big)^\sigma, \notag
\end{align}
The convergence of $(f_j)$ to $f$, the weak lower semicontinuity of $f_j$ and the minimality of $u_j$ give that
\begin{align*}
 \int_{B_R}  [f(x,Dv) -g \cdot v]dx    \leq &\int_{B_R}  [f(x,Dv) + \text{arctan}(|v-u|^2)  -g \cdot v]dx \\
 \le & \liminf_{j_0} \int_{B_R} [f_{j_0}(x,Dv) + \text{arctan}(|v-u|^2)  -g \cdot v]dx \notag\\
 \le & \liminf_{j_0} \liminf_j \int_{B_R} [f_{j_0}(x,Du_j) + \text{arctan}(|u_j-u|^2)  -g \cdot u_j]dx \notag\\
 \le &  \liminf_j \int_{B_R} [f_{j}(x,Du_j) + \text{arctan}(|u_j-u|^2)  -g \cdot u_j]dx \notag\\
  \le &  \liminf_j \int_{B_R} [f_{j}(x,Du) -g \cdot u]dx \notag\\
    \le & \int_{B_R} [f(x,Du)-g \cdot u]dx,
\end{align*}
then $v \in W^{1,p}_0(B_R, \mathbb{R}^N)+u$ is a local minimizer of \eqref{functional}.\\
Now we show that $v=u$ almost everywhere in $B_R$. Using the convergence $f_j \to f$, the weak lower semicontinuity of  $f_j$ and the minimality of $u_j$, we derive
\begin{align*}
 \int_{B_R}  [f(x,Dv) -g \cdot v]dx    \leq &\int_{B_R}  [f(x,Dv) + \text{arctan}(|v-u|^2)  -g \cdot v]dx \\
 \le & \liminf_{j_0} \int_{B_R} [f_{j_0}(x,Dv) + \text{arctan}(|v-u|^2)  -g \cdot v]dx \notag\\
 \leq &  \liminf_{j_0} \ \liminf_{j}\int_{B_R} [f_{j_0}(x,Du_j) + \text{arctan}(|u_j-u|^2)  -g \cdot u_j]dx \notag\\
 \leq &  \liminf_{j}\int_{B_R} [f_j(x,Du_j) + \text{arctan}(|u_j-u|^2)  -g \cdot u_j]dx \notag\\
  \leq &  \liminf_{j}\int_{B_R} [f_j(x,Du)    -g \cdot u]dx \notag\\
    \le & \int_{B_R} [f(x,Du)-g \cdot u]dx =  \int_{B_R}  [f(x,Dv) -g \cdot v]dx.
\end{align*}
Since $f$ is nonnegative the previous equality implies that
\begin{equation*}
    \int_{B_R}  \text{arctan}(|v-u|^2) dx = 0
\end{equation*}
and so $u=v$ a.e. in $B_R.$
\end{proof}

\vskip0.5cm

\noindent {{\bf Acknowledgements.} The authors are members of the Gruppo Nazionale per l’Analisi Matematica,
la Probabilità e le loro Applicazioni (GNAMPA) of the Istituto Nazionale di Alta Matematica (INdAM). A.G. Grimaldi has been partially supported through the INdAM$-$GNAMPA 2024 Project “Interazione ottimale tra la regolarità dei coefficienti e l’anisotropia del problema in funzionali integrali a crescite non standard” (CUP: E53C23001670001). In addition, A.G. Grimaldi has also been supported through the project: Geometric-Analytic Methods for PDEs and Applications (GAMPA) - funded by European Union - Next Generation EU within the PRIN 2022 program (D.D. 104 - 02/02/2022 Ministero dell'Università e della Ricerca). This manuscript reflects only the authors' views and opinions and the Ministry cannot be considered responsible for them.

\end{document}